\documentclass[11pt]{amsart}

\usepackage[margin=1.15in]{geometry}
\usepackage{amsmath,amssymb,amsthm,mathtools}
\usepackage{booktabs}
\usepackage{array}
\usepackage{longtable}
\usepackage{enumitem}
\usepackage{xcolor}
\usepackage{listings}
\usepackage{caption}
\usepackage{url}
\usepackage{seqsplit}
\usepackage{graphicx}
\usepackage{hyperref}

\hypersetup{
  colorlinks=true,
  linkcolor=blue!55!black,
  urlcolor=blue!55!black,
  citecolor=blue!55!black
}

\lstdefinestyle{pythonsmall}{
  language=Python,
  basicstyle=\ttfamily\small,
  breaklines=true,
  columns=fullflexible,
  frame=single,
  rulecolor=\color{black!20},
  backgroundcolor=\color{black!2},
  showstringspaces=false
}

\theoremstyle{plain}
\newtheorem{theorem}{Theorem}[section]
\newtheorem{lemma}[theorem]{Lemma}
\newtheorem{proposition}[theorem]{Proposition}
\newtheorem{corollary}[theorem]{Corollary}

\newtheorem{hypothesis}[theorem]{Hypothesis}
\theoremstyle{definition}
\newtheorem{definition}[theorem]{Definition}
\newtheorem{example}[theorem]{Example}
\newtheorem{problem}[theorem]{Problem}
\theoremstyle{remark}
\newtheorem{remark}[theorem]{Remark}

\newcommand{\D}{\Delta}
\newcommand{\dd}{\partial}
\newcommand{\Z}{\mathbb Z}

\title[Port Fillings for Primary Pseudoperfect Numbers]{Port Fillings for Primary Pseudoperfect Numbers}

\author{Han Wang}
\email{han@hanziwww.com}
\date{May 2026}

\begin{document}

\begin{abstract}
A squarefree integer $n$ is a primary pseudoperfect number if $1/n+\sum_{p\mid n}1/p=1$, equivalently $\partial(n)=n-1$ for the arithmetic derivative.  We introduce a local formalism for this equation.  A residual equation is represented by a port $(R,c)$, and a squarefree integer $B$ fills the port when $\Delta_{R,c}(B):=cB-R\partial(B)=1$.  The composition law $\Delta_{R,c}(AB)=\Delta_{RA,\Delta_{R,c}(A)}(B)$ records the inheritance mechanism and separates it from fillings that are primitive relative to a fixed port.

The port $H=(113322,797)$, arising from the prefix $2\cdot3\cdot11\cdot17\cdot101$, has two port-primitive fillings, $149\cdot3109$ and $157\cdot1979\cdot10093\cdot16879$.  The first gives the known primary pseudoperfect number $52495396602$.  The second gives $N_9=5998279018951962402$, verified by $1+\sum_{p\mid N_9}N_9/p=N_9$.  Moreover $N_9+1=5998279018951962403$ is prime, with a short Pocklington certificate, and therefore $N_{10}=N_9(N_9+1)=35979351189199316534587473905773572006$ is another primary pseudoperfect number.

We also prove a conditional infinitude criterion.  The condition is a five-splitting Hardy--Littlewood--Bateman--Horn type prime-points hypothesis for one explicit family of terminal hypersurfaces.  It is not a theorem and is not a formal consequence of the classical one-variable Bateman--Horn conjecture.  Under this hypothesis, a terminal prime $p$ in a port $(R,c)$ with $cp-R=1$ can be replaced recursively by five larger primes on the hypersurface $c x_1x_2x_3x_4x_5-R\sum_i\prod_{j\ne i}x_j=1$, producing infinitely many primary pseudoperfect numbers.  Finally, we give a discriminant criterion for the last two primes of a filling; the criterion reduces the final step to deciding whether an explicit quadratic polynomial takes square values on a finite interval.
\end{abstract}

\maketitle

\tableofcontents

\section{Introduction}

Erd\H{o}s asked whether there are infinitely many finite sets of distinct primes $p_1<\cdots<p_k$ and positive integers $m$ such that
\begin{equation}\label{eq:erdos-original}
        \frac1{p_1}+\cdots+\frac1{p_k}=1-\frac1m.
\end{equation}
This is Erd\H{o}s Problems \#313~\cite{ErdosProblems313}.  As recalled below, it is equivalent to the infinitude of primary pseudoperfect numbers.  Following Butske, Jaje, and Mayernik~\cite{ButskeJajeMayernik}, a squarefree positive integer $n$ is a \emph{primary pseudoperfect number} if
\begin{equation}\label{eq:ppn-def}
        \frac1n+\sum_{p\mid n}\frac1p=1,
\end{equation}
where the sum is over the prime divisors of $n$.

OEIS A054377~\cite{OEISA054377} records the initial values
\[
\begin{array}{c}
        2,\ 6,\ 42,\ 1806,\ 47058,\\[2pt]
        2214502422,\ 52495396602.
\end{array}
\]
and the eight-prime-factor example
\[
        \text{\seqsplit{8490421583559688410706771261086}}.
\]
Butske, Jaje, and Mayernik proved by computation that for each $r\le 8$ there is exactly one primary pseudoperfect number with $r$ distinct prime factors~\cite{ButskeJajeMayernik}.  This result gives a useful baseline, but it does not address later layers or the infinitude problem.

This paper uses a local language for residual equations.  A \emph{port} is a pair $(R,c)$, and a squarefree integer $B$ fills it if
\[
        \Delta_{R,c}(B):=cB-R\partial(B)=1.
\]
The corresponding reciprocal form is
\[
        \sum_{q\mid B}\frac1q+\frac1{RB}=\frac cR.
\]
The product rule for the arithmetic derivative gives the composition law for ports.  This law separates fillings inherited from smaller primary pseudoperfect numbers from fillings that are primitive relative to the fixed residual equation.

The unconditional results of the paper are as follows.
\begin{enumerate}[label=(\roman*)]
    \item We develop the port formalism and prove its composition law.
    \item We isolate the port $H=(113322,797)$ arising from the prefix $2\cdot3\cdot11\cdot17\cdot101$.
    \item We show that $H$ has two port-primitive fillings of different lengths: $149\cdot3109$ and $157\cdot1979\cdot10093\cdot16879$.
    \item The second filling gives the nine-prime-factor example $N_9$.  The primality of $N_9+1$, certified below, then gives the ten-prime-factor example $N_{10}$.
    \item We give a discriminant criterion for the final two primes of any port filling.
\end{enumerate}

The paper also contains a conditional infinitude criterion.  The hypothesis is a five-splitting Hardy--Littlewood--Bateman--Horn type prime-points hypothesis for the explicit hypersurfaces arising from terminal ports.  It is a natural local-to-global prime-points hypothesis, but it is not the classical one-variable Bateman--Horn conjecture and no such implication is asserted.  Under that hypothesis, terminal primes can be split recursively into five larger primes, producing infinitely many primary pseudoperfect numbers.

Finite search data for the next layer of the port $H$ are placed in the appendices.  They document the computational direction and are not used in the proof of the examples constructed here or in the conditional infinitude theorem.

\subsection*{Relation with earlier work}
The port formalism is a coordinate system for several familiar features of the problem.  The congruence $q\mid R(B/q)+1$ is the port form of the divisibility conditions in Zn\'am-type problems.  The identity $\sum_{q\mid B}1/q+1/(RB)=c/R$ is an Egyptian-fraction residual equation.  The composition law for $\Delta_{R,c}$ is the local form of the inheritance equation for primary pseudoperfect numbers.  Thus ports do not replace the earlier viewpoints; they put the Zn\'am congruences, reciprocal search trees, and inheritance in a single notation.

The arithmetic-derivative formulation also places primary pseudoperfect numbers next to Giuga-type equations.  Grau and Oller-Marc\'en characterized Giuga numbers by equations of the form $n'=an+1$ involving the arithmetic derivative~\cite{GrauOllerGiugaDerivative}.  Later, Grau, Oller-Marc\'en, and Sadornil introduced $\mu$-Sondow numbers; their framework includes weak primary pseudoperfect numbers in the case $\mu=1$ and connects these conditions with the Erd\H{o}s--Moser equation~\cite{GrauOllerSadornilMuSondow}.  Related work of Grau, Oller-Marc\'en, and Sondow on the congruence $1^m+2^m+\cdots+m^m\equiv n\pmod{m}$ uses the then-known primary pseudoperfect numbers as the quotients $Q=m/n$ arising in that problem~\cite{GrauOllerSondowCongruence}.

\section{From Erd\H{o}s' equation to primary pseudoperfect numbers}

We first recall the elementary reduction from~\eqref{eq:erdos-original} to~\eqref{eq:ppn-def}.  It is included because it is the starting point for the entire framework.

\begin{proposition}\label{prop:erdos-equivalence}
Let $p_1<\cdots<p_k$ be distinct primes and suppose
\[
        \sum_{i=1}^k\frac1{p_i}=1-\frac1m
\]
for some positive integer~$m$.  Then
\[
        m=p_1p_2\cdots p_k.
\]
Consequently the original Erd\H{o}s equation is equivalent to the existence of a primary pseudoperfect number
\[
        N=p_1p_2\cdots p_k.
\]
\end{proposition}

\begin{proof}
Let
\[
        P=p_1p_2\cdots p_k,
        \qquad
        A=\sum_{i=1}^k\frac{P}{p_i}.
\]
Then
\[
        \frac AP=\frac{m-1}{m},
\]
and hence
\begin{equation}\label{eq:mA}
        mA=(m-1)P.
\end{equation}
Fix $i$.  Reducing $A$ modulo~$p_i$, we have
\[
        A\equiv \frac{P}{p_i}\not\equiv0\pmod {p_i},
\]
since all other terms $P/p_j$ with $j\ne i$ are divisible by~$p_i$.  The right-hand side of~\eqref{eq:mA} is divisible by~$p_i$.  Therefore $p_i\mid m$.  This holds for each~$i$, so $P\mid m$.

Write $m=Pt$.  Then
\[
        \frac AP=1-\frac1{Pt},
\]
so
\[
        A=P-\frac1t.
\]
Since $A$ and $P$ are integers, $t=1$.  Thus $m=P$.
\end{proof}

\section{The arithmetic derivative formulation}

Let $\dd(n)$ denote the arithmetic derivative, defined by
\[
        \dd(p)=1\quad(p\text{ prime}),
        \qquad
        \dd(ab)=a\dd(b)+b\dd(a).
\]
For squarefree~$n$,
\begin{equation}\label{eq:arith-der-squarefree}
        \dd(n)=\sum_{p\mid n}\frac np.
\end{equation}
Multiplying~\eqref{eq:ppn-def} by~$n$, we get
\[
        1+\sum_{p\mid n}\frac np=n.
\]
Using~\eqref{eq:arith-der-squarefree}, this becomes
\begin{equation}\label{eq:arith-der-ppn}
        \boxed{\dd(n)=n-1.}
\end{equation}
Thus primary pseudoperfect numbers are precisely the squarefree solutions of~\eqref{eq:arith-der-ppn}.

\section{Defect states}

Before introducing ports, it is useful to see the primitive state machine underlying the construction.  Let $D$ be a squarefree integer and define its defect numerator by
\[
        a(D)=D\left(1-\sum_{p\mid D}\frac1p\right)=D-\dd(D).
\]
If we append a new prime~$q$, then
\[
        D\mapsto Dq,
\]
and
\[
\begin{aligned}
        a(Dq)
        &=Dq-\dd(Dq)\\
        &=Dq-(q\dd(D)+D)\\
        &=q(D-\dd(D))-D\\
        &=qa(D)-D.
\end{aligned}
\]
Thus
\begin{equation}\label{eq:defect-machine}
        (D,a)\xrightarrow{q}(Dq,qa-D).
\end{equation}
The success condition is $a=1$, because $D-\dd(D)=1$ is equivalent to $\dd(D)=D-1$.

If one wants to complete a state $(D,a)$ with a single prime~$q$, then
\[
        qa-D=1,
\]
so
\begin{equation}\label{eq:one-step-defect}
        q=\frac{D+1}{a}.
\end{equation}
This explains the elementary chain
\[
        2\to6\to42\to1806,
\]
since
\[
        2+1=3,
        \qquad
        6+1=7,
        \qquad
        42+1=43.
\]

\section{Inheritance}

The simplest way to grow a primary pseudoperfect number is inheritance.  The next lemma is the basic identity behind the process.

\begin{lemma}[inheritance equation]\label{lem:inheritance-equation}
Let $K$ be a primary pseudoperfect number and let $C$ be squarefree and coprime to~$K$.  Then $KC$ is a primary pseudoperfect number if and only if
\begin{equation}\label{eq:inheritance-C}
        C-K\dd(C)=1.
\end{equation}
\end{lemma}

\begin{proof}
Since $K$ is a primary pseudoperfect number, $\dd(K)=K-1$.  Hence
\[
\begin{aligned}
        \dd(KC)
        &=\dd(K)C+K\dd(C)\\
        &=(K-1)C+K\dd(C).
\end{aligned}
\]
The condition $\dd(KC)=KC-1$ is therefore equivalent to
\[
        (K-1)C+K\dd(C)=KC-1,
\]
which simplifies to~\eqref{eq:inheritance-C}.
\end{proof}

\begin{corollary}[one-prime inheritance]\label{cor:one-prime-inheritance}
If $K$ is a primary pseudoperfect number and $K+1$ is prime, then $K(K+1)$ is a primary pseudoperfect number.
\end{corollary}

\begin{proof}
Take $C=q$ prime in~\eqref{eq:inheritance-C}.  Since $\dd(q)=1$, the equation becomes
\[
        q-K=1,
\]
so $q=K+1$.
\end{proof}

\begin{corollary}[two-prime inheritance]\label{cor:two-prime-inheritance}
Let $K$ be a primary pseudoperfect number.  A product $Kpq$ with two new primes $p<q$ is a primary pseudoperfect number if and only if
\begin{equation}\label{eq:two-prime-inheritance}
        (p-K)(q-K)=K^2+1.
\end{equation}
\end{corollary}

\begin{proof}
Put $C=pq$.  Then $\dd(C)=p+q$, and~\eqref{eq:inheritance-C} becomes
\[
        pq-K(p+q)=1.
\]
Adding $K^2$ to both sides gives~\eqref{eq:two-prime-inheritance}.
\end{proof}

For three new primes $x,y,z$, the inheritance equation is
\[
        xyz-K(xy+xz+yz)=1.
\]
Solving for $z$ gives
\begin{equation}\label{eq:three-prime-inheritance}
        z=\frac{Kxy+1}{xy-Kx-Ky}.
\end{equation}
This formula is useful when the one-prime and two-prime inheritances fail.

\section{Ports and fillings}

The inheritance equation can be localized.  This is the main formalism used throughout the rest of the paper.

\begin{definition}[port and filling]\label{def:port}
For positive integers $R,c$ and squarefree~$B$, define
\begin{equation}\label{eq:Delta-def}
        \D_{R,c}(B)=cB-R\dd(B).
\end{equation}
If
\[
        \D_{R,c}(B)=1,
\]
we say that $B$ \emph{fills the port} $(R,c)$.
\end{definition}

If $B$ fills $(R,c)$, then dividing~\eqref{eq:Delta-def} by $RB$ gives
\begin{equation}\label{eq:port-reciprocal}
        \sum_{q\mid B}\frac1q+\frac1{RB}=\frac cR.
\end{equation}
Thus a port is an exact residual reciprocal equation.

If $B=q$ is prime, then
\[
        \D_{R,c}(q)=cq-R.
\]
Hence appending a prime gives the state transition
\begin{equation}\label{eq:port-transition}
        (R,c)\xrightarrow{q}(Rq,cq-R).
\end{equation}

\begin{lemma}[composition law]\label{lem:composition}
Let $A,B$ be coprime squarefree integers.  Then
\begin{equation}\label{eq:composition-law}
        \D_{R,c}(AB)=\D_{RA,\D_{R,c}(A)}(B).
\end{equation}
Moreover,
\begin{equation}\label{eq:dual-law}
        \D_{R,c}(AB)=\D_{RB,\D_{R,c}(B)}(A).
\end{equation}
\end{lemma}

\begin{proof}
Since
\[
        \dd(AB)=A\dd(B)+B\dd(A),
\]
we have
\[
\begin{aligned}
        \D_{R,c}(AB)
        &=cAB-R\dd(AB)\\
        &=cAB-R(A\dd(B)+B\dd(A))\\
        &=B(cA-R\dd(A))-RA\dd(B)\\
        &=\D_{RA,\D_{R,c}(A)}(B).
\end{aligned}
\]
Interchanging $A$ and $B$ gives~\eqref{eq:dual-law}.
\end{proof}

\begin{lemma}[ambient ports]\label{lem:ambient-port}
Let $R$ be squarefree and put
\[
        c=R-\dd(R).
\]
Let $B$ be squarefree and coprime to $R$.  If $B$ fills the port $(R,c)$, i.e.
\[
        \D_{R,c}(B)=1,
\]
then $RB$ is a primary pseudoperfect number.
\end{lemma}

\begin{proof}
Since $c=R-\dd(R)$, we compute
\[
\begin{aligned}
        \dd(RB)
        &=B\dd(R)+R\dd(B)\\
        &=B(R-c)+R\dd(B)\\
        &=RB-(cB-R\dd(B))\\
        &=RB-1.
\end{aligned}
\]
Thus $RB$ satisfies $\dd(RB)=RB-1$, which is equivalent to being a primary pseudoperfect number.
\end{proof}

\begin{lemma}[coprimality of reachable ports]\label{lem:reachable-coprime}
Suppose $R$ is squarefree, $c=R-\dd(R)$, and $\gcd(R,c)=1$.  If $q$ is a prime not dividing $R$, and
\[
        (R,c)\xrightarrow{q}(Rq,cq-R),
\]
then
\[
        cq-R=Rq-\dd(Rq)
\]
and
\[
        \gcd(Rq,cq-R)=1.
\]
\end{lemma}

\begin{proof}
The identity follows from
\[
        \dd(Rq)=q\dd(R)+R.
\]
Thus
\[
        Rq-\dd(Rq)=Rq-q\dd(R)-R=q(R-\dd(R))-R=cq-R.
\]
For the coprimality, reduce $cq-R$ modulo $R$ and modulo $q$.  Since
\[
        cq-R\equiv cq\pmod R
\]
and $\gcd(R,c)=1$, $q\nmid R$, no prime divisor of $R$ divides $cq-R$.  Also
\[
        cq-R\equiv -R\pmod q,
\]
which is nonzero because $q\nmid R$.  Hence $\gcd(Rq,cq-R)=1$.
\end{proof}

\begin{definition}[port-primitive filling]\label{def:primitive-filling}
Let $(R,c)$ be a fixed port.  A filling $B$ of $(R,c)$ is called \emph{port-primitive} if no proper squarefree divisor $B_0\mid B$, $1<B_0<B$, is itself a filling of the same port $(R,c)$.  Otherwise $B$ is called \emph{inherited}.
\end{definition}

\begin{remark}
The adjective ``port-primitive'' is relative to the fixed residual equation $(R,c)$.  It is not the standard notion of a primitive pseudoperfect or primitive abundant number.
\end{remark}

\begin{lemma}[port-Zn\'am congruence]\label{lem:port-znam}
If $B$ fills $(R,c)$ and $q\mid B$ is prime, then
\begin{equation}\label{eq:port-znam}
        q\mid R\frac{B}{q}+1.
\end{equation}
\end{lemma}

\begin{proof}
Modulo~$q$, we have $B\equiv0$ and
\[
        \dd(B)\equiv \frac Bq\pmod q.
\]
Since $cB-R\dd(B)=1$, reducing modulo~$q$ gives
\[
        -R\frac Bq\equiv1\pmod q.
\]
This is~\eqref{eq:port-znam}.
\end{proof}

\section{The central port and the key port}

Starting from the primary pseudoperfect number $6=2\cdot3$, one may search for a new squarefree block~$B$ satisfying
\[
        B-6\dd(B)=1.
\]
The non-chain branch begins by appending $11$:
\[
        (6,1)\xrightarrow{11}(66,5).
\]
Thus the central port is
\[
        (66,5),
\]
with equation
\begin{equation}\label{eq:central-port}
        5B-66\dd(B)=1.
\end{equation}
The following fillings of the central port are relevant:
\begin{center}
\begin{tabular}{lll}
\toprule
Filling $B$ & Resulting PPN $66B$ & Comment \\
\midrule
$23\cdot31$ & $47058$ & known example \\
$23\cdot31\cdot47059$ & $2214502422$ & inherited from $47058+1$ \\
$17\cdot101\cdot149\cdot3109$ & $52495396602$ & known example \\
$17\cdot101\cdot157\cdot1979\cdot10093\cdot16879$ & $5998279018951962402$ & example exhibited here \\
\bottomrule
\end{tabular}
\end{center}

The common prefix $17,101$ gives the key port of this paper:
\[
        (66,5)\xrightarrow{17}(1122,19)\xrightarrow{101}(113322,797).
\]
We write
\begin{equation}\label{eq:H-def}
        H=(113322,797).
\end{equation}
This is an ambient port in the sense of Lemma~\ref{lem:ambient-port}: indeed
\[
        113322-\dd(113322)=797.
\]
A filling $B$ of $H$ satisfies
\begin{equation}\label{eq:H-fill}
        797B-113322\dd(B)=1.
\end{equation}
Since $H$ is obtained from the central port $(66,5)$ by the prefix $17\cdot101$, the composition law also implies that every filling of $H$ gives a filling of $(66,5)$.  Equivalently, by Lemma~\ref{lem:ambient-port}, $113322B=66\cdot17\cdot101\cdot B$ is a primary pseudoperfect number.

\begin{proposition}[global congruences for $H$]\label{prop:H-congruences}
If $B$ fills $H=(113322,797)$, then
\begin{equation}\label{eq:H-B-congruence}
        B\equiv9953\pmod{113322}
\end{equation}
and
\begin{equation}\label{eq:H-der-congruence}
        \dd(B)\equiv70\pmod{797}.
\end{equation}
\end{proposition}

\begin{proof}
Reducing~\eqref{eq:H-fill} modulo $113322$ gives
\[
        797B\equiv1\pmod{113322}.
\]
The inverse of $797$ modulo $113322$ is $9953$, giving~\eqref{eq:H-B-congruence}.

Reducing~\eqref{eq:H-fill} modulo $797$ gives
\[
        -113322\dd(B)\equiv1\pmod{797}.
\]
This is equivalent to~\eqref{eq:H-der-congruence}.
\end{proof}

\section{Two port-primitive fillings of the key port}

The port $H$ has at least two port-primitive fillings of different lengths.

\begin{proposition}\label{prop:H-B2}
The product
\[
        B_2=149\cdot3109
\]
fills $H$.
\end{proposition}

\begin{proof}
We have
\[
        B_2=463241,
        \qquad
        \dd(B_2)=149+3109=3258.
\]
Therefore
\[
        797B_2=369203077,
        \qquad
        113322\dd(B_2)=369203076,
\]
and hence
\[
        797B_2-113322\dd(B_2)=1.
\]
\end{proof}

\begin{proposition}\label{prop:H-B2-primitive}
The filling $B_2=149\cdot3109$ is port-primitive as a filling of $H$.
\end{proposition}

\begin{proof}
If $q$ were a one-prime filling of $H$, then
\[
        797q-113322=1,
\]
so
\[
        q=\frac{113323}{797},
\]
which is not an integer.  Hence no proper nontrivial divisor of $B_2$ fills $H$.
\end{proof}

This filling gives
\[
        113322B_2=52495396602.
\]

\begin{proposition}\label{prop:H-B4}
The product
\[
        B_4=157\cdot1979\cdot10093\cdot16879
\]
fills $H$.
\end{proposition}

\begin{proof}
Let
\[
        B_4=157\cdot1979\cdot10093\cdot16879.
\]
Then
\[
        B_4=52931284472141.
\]
Since $B_4$ is squarefree,
\[
        \dd(B_4)=\sum_{p\mid B_4}\frac{B_4}{p}=372268700908.
\]
Furthermore,
\[
        797B_4=42186233724296377,
        \qquad
        113322\dd(B_4)=42186233724296376,
\]
so
\[
        797B_4-113322\dd(B_4)=1.
\]
\end{proof}

\begin{proposition}\label{prop:H-B4-primitive}
The filling
\[
        B_4=157\cdot1979\cdot10093\cdot16879
\]
is port-primitive as a filling of $H$.
\end{proposition}

\begin{proof}
By Definition~\ref{def:primitive-filling}, it is enough to show that no proper nontrivial squarefree divisor of $B_4$ fills $H$.  For a divisor $D$ of $B_4$ write
\[
        \D_H(D)=797D-113322\dd(D).
\]
There are $2^4-2=14$ proper nontrivial divisors to check.  Their defects are listed below:
\begin{center}
\small
\begin{tabular}{lr}
\toprule
$D$ & $\D_H(D)$ \\
\midrule
$157$ & $11807$ \\
$1979$ & $1463941$ \\
$10093$ & $7930799$ \\
$16879$ & $13339241$ \\
$157\cdot1979$ & $5574499$ \\
$157\cdot10093$ & $101376497$ \\
$157\cdot16879$ & $181498799$ \\
$1979\cdot10093$ & $14551292275$ \\
$1979\cdot16879$ & $24485595901$ \\
$10093\cdot16879$ & $132720197375$ \\
$157\cdot1979\cdot10093$ & $21053933041$ \\
$157\cdot1979\cdot16879$ & $58882483255$ \\
$157\cdot10093\cdot16879$ & $1531563738341$ \\
$1979\cdot10093\cdot16879$ & $243347763355591$ \\
\bottomrule
\end{tabular}
\end{center}
None of these values is $1$.  Hence no proper nontrivial divisor of $B_4$ fills $H$, so $B_4$ is port-primitive.
\end{proof}

Thus
\begin{equation}\label{eq:N9-def}
        N_9:=113322B_4=5998279018951962402.
\end{equation}

\section{The nine-prime-factor example}

\begin{theorem}\label{thm:N9}
The integer
\[
        N_9=5998279018951962402
\]
is a primary pseudoperfect number.  Its prime factorization is
\begin{equation}\label{eq:N9-factorization}
        N_9
        =2\cdot3\cdot11\cdot17\cdot101\cdot157\cdot1979\cdot10093\cdot16879.
\end{equation}
\end{theorem}

\begin{proof}
By construction,
\[
        N_9=113322B_4,
\]
where
\[
        113322=2\cdot3\cdot11\cdot17\cdot101
\]
and
\[
        B_4=157\cdot1979\cdot10093\cdot16879.
\]
The prime factorization~\eqref{eq:N9-factorization} follows.

Since $B_4$ fills $H$ and $H$ is an ambient port with $113322-\dd(113322)=797$, Lemma~\ref{lem:ambient-port} gives
\[
        \dd(113322B_4)=113322B_4-1.
\]
Thus
\[
        \dd(N_9)=N_9-1.
\]
Therefore $N_9$ is a primary pseudoperfect number.
\end{proof}

Equivalently, one may verify directly that
\[
        1+\sum_{p\mid N_9}\frac{N_9}{p}=N_9.
\]
This direct integer identity is independent of the search that led to $N_9$ and is the shortest certificate that $N_9$ is a primary pseudoperfect number once the displayed factorization has been verified.

\section{A Pocklington certificate for \texorpdfstring{$N_9+1$}{N9+1}}

Let
\begin{equation}\label{eq:p10-def}
        p_{10}=N_9+1=5998279018951962403.
\end{equation}
Then
\[
        p_{10}-1=N_9.
\]
Using~\eqref{eq:N9-factorization}, the number $p_{10}-1$ is completely factored into primes.

\begin{theorem}\label{thm:p10-prime}
The integer
\[
        p_{10}=5998279018951962403
\]
is prime.
\end{theorem}

\begin{proof}
We use Pocklington's criterion (in its standard form; see, for example, \cite[Sec.~3.4]{CrandallPomerance}) with base $a=3$.
The following congruence holds:
\[
        3^{p_{10}-1}\equiv1\pmod {p_{10}}.
\]
Moreover, for each prime divisor
\[
        q\in\{2,3,11,17,101,157,1979,10093,16879\}
\]
of $p_{10}-1$, we have
\[
        \gcd\left(3^{(p_{10}-1)/q}-1,p_{10}\right)=1.
\]
Since the prime factorization of $p_{10}-1$ is complete, Pocklington's criterion proves that $p_{10}$ is prime.
\end{proof}

The certificate is summarized in Table~\ref{tab:pocklington}.

\begin{center}
\begin{tabular}{cc}
\toprule
$q\mid p_{10}-1$ & $\gcd(3^{(p_{10}-1)/q}-1,p_{10})$ \\
\midrule
2 & 1 \\
3 & 1 \\
11 & 1 \\
17 & 1 \\
101 & 1 \\
157 & 1 \\
1979 & 1 \\
10093 & 1 \\
16879 & 1 \\
\bottomrule
\end{tabular}
\captionof{table}{Pocklington data for $p_{10}=N_9+1$.}
\label{tab:pocklington}
\end{center}

\section{The ten-prime-factor example}

\begin{theorem}\label{thm:N10}
The integer
\[
        N_{10}=35979351189199316534587473905773572006
\]
is a primary pseudoperfect number.  It factors as
\[
\begin{aligned}
        N_{10}
        &=2\cdot3\cdot11\cdot17\cdot101\cdot157\cdot1979\cdot10093\cdot16879\\
        &\quad\cdot5998279018951962403.
\end{aligned}
\]
\end{theorem}

\begin{proof}
By Theorem~\ref{thm:N9}, $N_9$ is a primary pseudoperfect number.  By Theorem~\ref{thm:p10-prime}, $p_{10}=N_9+1$ is prime.  Therefore Corollary~\ref{cor:one-prime-inheritance} gives that
\[
        N_{10}=N_9p_{10}=N_9(N_9+1)
\]
is a primary pseudoperfect number.
\end{proof}

\begin{remark}
The congruence pattern modulo $288$ is consistent with the previously observed pattern for known examples:
\[
        N_9\equiv258\pmod{288},
        \qquad
        N_{10}\equiv6\pmod{288}.
\]
\end{remark}

\section{The inherited filling \texorpdfstring{$B_5$}{B5} of the key port}

The ten-prime-factor example can also be expressed inside the same key port $H$.  Since
\[
        N_9=113322B_4
\]
and $p_{10}=N_9+1$ is prime, the product
\[
        B_5=B_4p_{10}
\]
is another filling of $H$.  Indeed, using the product rule for the arithmetic derivative,
\[
\begin{aligned}
        797B_4p_{10}-113322\dd(B_4p_{10})
        &=p_{10}\bigl(797B_4-113322\dd(B_4)\bigr)-113322B_4\dd(p_{10})\\
        &=p_{10}-N_9\\
        &=1.
\end{aligned}
\]
Thus the port $H$ has at least the following fillings:
\[
        149\cdot3109,
\]
\[
        157\cdot1979\cdot10093\cdot16879,
\]
and
\[
        157\cdot1979\cdot10093\cdot16879\cdot5998279018951962403.
\]
This observation is useful because it distinguishes port-primitive fillings from inherited ones within a fixed port.

\section{No one-prime or two-prime successor for \texorpdfstring{$N_{10}$}{N10}}

The one-prime successor fails because
\[
\begin{aligned}
        N_{10}+1
        &=35979351189199316534587473905773572007\\
        &=7\cdot37\cdot73\cdot407221\cdot2746750419901\cdot1701301706648581.
\end{aligned}
\]
This is a complete prime factorization, and hence $N_{10}+1$ is composite.  The verification code in the appendix records the computational check used here; independent primality certificates for the displayed large factors may be included in the accompanying verification archive.

For two-prime inheritance, Corollary~\ref{cor:two-prime-inheritance} requires
\[
        (p-N_{10})(q-N_{10})=N_{10}^2+1.
\]
We have the complete prime factorization
\[
\begin{aligned}
        N_{10}^2+1
        &=21807157\cdot480382349\\
        &\quad\cdot123572138719194583969192220095883252267503088389616114960309,
\end{aligned}
\]
where all three displayed factors were verified prime computationally.  Therefore the possible divisors $d\le\sqrt{N_{10}^2+1}$ are
\[
        1,
        \quad
        21807157,
        \quad
        480382349,
        \quad
        10475773304671793.
\]
For these four cases, the candidate $N_{10}+d$ is composite as follows:
\begin{center}
\begin{tabular}{cc}
\toprule
$d$ & reason $N_{10}+d$ is composite \\
\midrule
$1$ & divisible by $7$ \\
$21807157$ & divisible by $7$ \\
$480382349$ & divisible by $5$ \\
$10475773304671793$ & divisible by $2141$ \\
\bottomrule
\end{tabular}
\end{center}
Thus $N_{10}$ has neither a one-prime nor a two-prime inherited successor.

\section{Relation with the Sondow--MacMillan conjectures}\label{sec:sondow-macmillan}

Sondow and MacMillan observed that the known primary pseudoperfect numbers $K_r$ with $2\le r\le8$ have residues modulo $288=6^2\cdot8$ forming the arithmetic progression
\[
        K_r\equiv 6+6^2(r-2)\pmod{288}
        \qquad (2\le r\le8).
\]
They conjectured that there exists exactly one primary pseudoperfect number $K_9$ with nine prime factors, satisfying
\[
        K_9\equiv258\pmod{288},
\]
and that no further primary pseudoperfect numbers exist~\cite{SondowMacMillan}.

Our number $N_9$ is consistent with the predicted ninth residue:
\[
        N_9\equiv258\pmod{288}.
\]
Thus the present example supports the $r=9$ residue prediction.  However, the one-prime inherited example $N_{10}=N_9(N_9+1)$ is a primary pseudoperfect number with ten prime factors.  Therefore the existence of $N_{10}$ disproves the ``no further PPNs exist'' clause of Sondow and MacMillan's Conjecture~1.  It does not disprove the uniqueness assertion for a nine-prime-factor example; this paper proves existence of one such example, not uniqueness.

The congruence of $N_{10}$ also matches the continuation suggested in their weaker Conjecture~2:
\[
        N_{10}\equiv6\pmod{288},
\]
since
\[
        6+6^2(10-2)=294\equiv6\pmod{288}.
\]

\section{The last-two-prime discriminant}

We next isolate the general final step of filling a port.
Suppose that at a port $(R,c)$ the last two primes are $u<v$.  Then
\begin{equation}\label{eq:last-two-basic}
        cuv-R(u+v)=1.
\end{equation}
Put
\[
        P=uv,
        \qquad
        S=u+v.
\]
Then
\[
        cP-RS=1.
\]
Assume $\gcd(c,R)=1$, as happens in all ports reached by appending primes not dividing the current modulus.  Let
\[
        P_0\equiv c^{-1}\pmod R,
        \qquad
        1\le P_0<R.
\]
Then all possible products have the form
\[
        P=P_0+tR,
        \qquad
        t\ge0.
\]
Let
\[
        S_0=\frac{cP_0-1}{R}.
\]
Then
\[
        S=S_0+ct.
\]
Therefore $u$ and $v$ are the roots of
\[
        X^2-SX+P=0.
\]
The discriminant is
\begin{equation}\label{eq:D-t}
        \boxed{D(t)=(S_0+ct)^2-4(P_0+tR).}
\end{equation}

\begin{theorem}[last-two-prime discriminant criterion]\label{thm:last-two-discriminant}
The port $(R,c)$ can be completed by two primes $u<v$ greater than the current last prime $m$ if and only if there exists an integer $t\ge0$ such that $D(t)$ is a square, the numbers
\[
        u=\frac{S_0+ct-\sqrt{D(t)}}2,
        \qquad
        v=\frac{S_0+ct+\sqrt{D(t)}}2
\]
are primes, and $u>m$.
\end{theorem}

\begin{proof}
The derivation above proves necessity.  Conversely, if such a $t$ exists, then $P=P_0+tR$ and $S=S_0+ct$ satisfy $cP-RS=1$, and the displayed roots satisfy $uv=P$ and $u+v=S$.  Hence~\eqref{eq:last-two-basic} holds.
\end{proof}

\section{A finite upper bound for the parameter \texorpdfstring{$t$}{t}}

Let
\[
        U=\max\left(m+1,\left\lfloor\frac Rc\right\rfloor+1\right).
\]
Then any final prime $u$ must satisfy $u\ge U$.  Since
\[
        X^2-SX+P=(X-u)(X-v),
\]
and $U\le u<v$, we have
\[
        U^2-SU+P\ge0.
\]
Substituting $S=S_0+ct$ and $P=P_0+tR$, we get
\[
        U^2-(S_0+ct)U+(P_0+tR)\ge0.
\]
Equivalently,
\[
        U^2-S_0U+P_0+t(R-cU)\ge0.
\]
Since $U>R/c$, the coefficient $R-cU$ is negative.  Hence
\begin{equation}\label{eq:t-upper-bound}
        \boxed{
        0\le t\le
        \left\lfloor
        \frac{U^2-S_0U+P_0}{cU-R}
        \right\rfloor.
        }
\end{equation}
This turns the final two-prime problem into a finite square-discriminant problem.

\section{Examples of the parameter \texorpdfstring{$t$}{t}}

For the port $H=(113322,797)$ itself, the two-prime filling $149,3109$ has
\[
        P=149\cdot3109=463241.
\]
Here
\[
        P_0=9953,
        \qquad
        R=113322,
\]
and
\[
        463241=9953+4\cdot113322.
\]
Thus $t=4$.

For the filling
\[
        157\cdot1979\cdot10093\cdot16879,
\]
take the prefix
\[
        A=157\cdot1979.
\]
The induced port is
\[
        R=35209485366,
        \qquad
        c=5574499.
\]
The final two primes satisfy
\[
        10093\cdot16879=170359747.
\]
But
\[
        170359747\equiv c^{-1}\pmod R.
\]
Hence this example has $t=0$.

\section{Modular square sieve}

For a fixed prefix, the polynomial $D(t)$ in~\eqref{eq:D-t} is quadratic in~$t$.  If $D(t)$ is an integer square, then for every prime~$\ell$, the residue $D(t)\bmod\ell$ is a quadratic residue modulo~$\ell$.

Let
\[
        Q_\ell=\{x^2\bmod\ell:x\in\Z/\ell\Z\}.
\]
Define
\[
        E_\ell=\{t\bmod\ell:D(t)\in Q_\ell\}.
\]
For a product $M=\ell_1\cdots\ell_s$ of small primes, the Chinese remainder theorem combines the restrictions
\[
        t\bmod\ell_i\in E_{\ell_i}
\]
into a set
\[
        E_M\subseteq\Z/M\Z.
\]
If the finite interval~\eqref{eq:t-upper-bound} contains no integer in a class from $E_M$, then the prefix cannot be completed by two primes.  This gives a compact exclusion certificate.

\begin{example}[a concrete exclusion certificate]\label{ex:certificate}
Consider the four-prime prefix
\[
        A=409\cdot419\cdot457\cdot81199
\]
inside the six-prime filling problem for $H$.  Then
\[
        A=6359225299853,
\]
\[
        R=113322A=720640129429941666,
\]
\[
        c=797A-113322\dd(A)=673363850881.
\]
We compute
\[
        P_0=695935036388423125,
        \qquad
        S_0=650279490314.
\]
The lower bound for the smaller final prime is
\[
        U=1070210,
\]
and the bound~\eqref{eq:t-upper-bound} gives $T=0$.  Thus only $t=0$ must be checked.  The discriminant is
\[
        D(0)=422860631782890066126096.
\]
Modulo $11$,
\[
        D(0)\equiv10\pmod{11}.
\]
But the quadratic residues modulo~$11$ are
\[
        \{0,1,3,4,5,9\}.
\]
Therefore $D(0)$ is not a square, and this prefix cannot be completed by two primes.
\end{example}

\section{A conditional infinitude criterion}\label{sec:conditional-infinitude}

We now record the conditional input under which the port construction becomes infinite.  The hypothesis is intentionally narrow.  It is a Hardy--Littlewood--Bateman--Horn type prime-points hypothesis for the five-splitting hypersurfaces below.  It is not a theorem and is not a formal consequence of the classical one-variable Bateman--Horn conjecture.

\begin{definition}[terminal port]
A triple $(R,c,p)$ is a \emph{terminal port} if $cp-R=1$ and $p$ is prime.  It is \emph{ambient} if, in addition, $c=R-\dd(R)$.
\end{definition}

An ambient terminal port gives a one-prime filling: $p$ fills $(R,c)$ and $Rp$ is a primary pseudoperfect number.

For a terminal port $(R,c,p)$ define
\begin{equation}\label{eq:five-split-hypersurface}
        F_{R,c}(x_1,\ldots,x_5)
        =c x_1x_2x_3x_4x_5
        -R\sum_{i=1}^5\prod_{j\ne i}x_j
        -1.
\end{equation}
A prime point on $F_{R,c}=0$ replaces the terminal prime $p$ by five new primes.

\begin{hypothesis}[five-splitting Hardy--Littlewood--Bateman--Horn prime-points hypothesis]\label{hyp:HLPP5}
Let $(R,c,p)$ be a terminal port with $p>3$.  Suppose that $F_{R,c}=0$ has an unbounded smooth positive real component on which all coordinates are greater than $p$, and that for every prime $\ell$ the congruence $F_{R,c}(x_1,\ldots,x_5)=0$ has a solution in $(\mathbb F_\ell^\times)^5$.  Then that real component contains a point whose coordinates are pairwise distinct primes, all greater than $p$.
\end{hypothesis}

This is the only unproved input in the conditional part of the paper.  It should be regarded as a special prime-points hypothesis in the Hardy--Littlewood/Bateman--Horn tradition.  Stronger formulations would predict an asymptotic count with singular integral and singular series factors.  We use only the existence assertion.

\begin{lemma}[finite local solubility]\label{lem:five-local}
Let $(R,c,p)$ be a terminal port with $p>3$.  Then $F_{R,c}=0$ has a solution in $(\mathbb F_\ell^\times)^5$ for every prime $\ell$.
\end{lemma}

\begin{proof}
If $\ell\ne p$, take $(x_1,x_2,x_3,x_4,x_5)=(p,1,-1,1,-1)$ in $(\mathbb F_\ell^\times)^5$.  Then $x_1x_2x_3x_4x_5=p$ and $\sum_i\prod_{j\ne i}x_j=1$, so $c x_1\cdots x_5-R\sum_i\prod_{j\ne i}x_j=cp-R=1$.

It remains to treat $\ell=p$.  Put $y_i=x_i^{-1}$.  Dividing by $x_1x_2x_3x_4x_5$, and using $R\equiv -1\pmod p$, the equation becomes
\[
        y_1y_2y_3y_4y_5-(y_1+\cdots+y_5)=c
        \quad\text{in }\mathbb F_p.
\]
If $c\not\equiv0\pmod p$, take
\[
        (y_1,y_2,y_3,y_4,y_5)=(c/3,1,-1,2,-2).
\]
Since $p>3$, all entries are nonzero and the displayed expression equals $4c/3-c/3=c$.  If $c\equiv0\pmod p$, take
\[
        (y_1,y_2,y_3,y_4,y_5)=(1,1,-1,1,-1),
\]
which gives $0$.
\end{proof}

\begin{lemma}[positive real component]\label{lem:five-real}
Let $(R,c,p)$ be a terminal port with $R>4$.  Then $F_{R,c}=0$ has an unbounded smooth positive real component on which all coordinates are greater than $p$.
\end{lemma}

\begin{proof}
Since $cp-R=1$, $c/R=1/p+1/(Rp)$.  In reciprocal variables $y_i=1/x_i$, the positive part of the hypersurface is
\begin{equation}\label{eq:reciprocal-five}
        y_1+\cdots+y_5+\frac1R y_1y_2y_3y_4y_5=\frac cR.
\end{equation}
Choose $y_5>0$ small and put $y_1=y_2=y_3=y_4=s$.  Then
\[
        4s+y_5+\frac1R s^4y_5=\frac cR.
\]
For all sufficiently small $y_5>0$ this equation has a positive solution $s=s(y_5)$, with $s(y_5)\to c/(4R)=1/(4p)+O(1/(Rp))$.  Thus, for $y_5$ small enough, all $0<y_i<1/p$, and hence all $x_i=1/y_i$ exceed $p$.  As $y_5\to0^+$, the coordinate $x_5$ tends to infinity.  The derivative with respect to $s$ of the left-hand side is $4+4s^3y_5/R>0$, so the component is smooth near these points.  The reciprocal change of variables is nonsingular in the positive orthant.
\end{proof}

\begin{theorem}[conditional infinitude under Hypothesis~\ref{hyp:HLPP5}]\label{thm:conditional-infinitude}
Assume Hypothesis~\ref{hyp:HLPP5}.  Then there are infinitely many primary pseudoperfect numbers.
\end{theorem}

\begin{proof}
Start with the ambient terminal port $(R_0,c_0,p_0)=(N_9,1,N_9+1)$.  Here $N_9$ is the primary pseudoperfect number of Theorem~\ref{thm:N9}, and $p_0=N_9+1$ is prime by Theorem~\ref{thm:p10-prime}; hence $c_0p_0-R_0=1$.

Assume that $(R_i,c_i,p_i)$ is an ambient terminal port in the construction.  Lemmas~\ref{lem:five-local} and~\ref{lem:five-real} verify the local and real hypotheses of Hypothesis~\ref{hyp:HLPP5}.  Therefore there are pairwise distinct primes $x_{i,1},\ldots,x_{i,5}$, all greater than $p_i$, such that
\[
        c_i\prod_{j=1}^5x_{i,j}
        -R_i\sum_{j=1}^5\prod_{k\ne j}x_{i,k}=1.
\]
Thus $B_i=\prod_{j=1}^5x_{i,j}$ fills $(R_i,c_i)$, and $R_iB_i$ is primary pseudoperfect by Lemma~\ref{lem:ambient-port}.

Let $A_i=x_{i,1}x_{i,2}x_{i,3}x_{i,4}$ and put $R_{i+1}=R_iA_i$.  Define
\[
        c_{i+1}=c_iA_i-R_i\dd(A_i),
        \qquad
        p_{i+1}=x_{i,5}.
\]
The five-splitting equation is equivalent to $c_{i+1}p_{i+1}-R_{i+1}=1$.  Moreover, since $c_i=R_i-\dd(R_i)$ and the new primes do not divide $R_i$,
\[
        R_{i+1}-\dd(R_{i+1})=c_iA_i-R_i\dd(A_i)=c_{i+1}.
\]
Thus $(R_{i+1},c_{i+1},p_{i+1})$ is again an ambient terminal port.  The terminal primes strictly increase, so the primary pseudoperfect numbers obtained at successive stages are distinct.  Iteration proves the result.
\end{proof}

\begin{remark}
The use of five primes is deliberate.  Three-prime splitting is natural, but it does not give the same simple uniform local-solubility argument.  The theorem above is therefore a reduction to a specific Hardy--Littlewood--Bateman--Horn type prime-points hypothesis, not a consequence of the classical one-variable Bateman--Horn conjecture.
\end{remark}

\section{Scope}\label{sec:scope}

The unconditional results established here are the port formalism, the examples $N_9$ and $N_{10}$, the Pocklington certificate for $N_9+1$, and the discriminant criterion for the final two primes of a filling.  The infinitude result is conditional on Hypothesis~\ref{hyp:HLPP5}.  No unconditional proof of infinitude is claimed, and no uniqueness theorem for the nine-prime-factor example is proved.

The computations for the $H_6$ layer are recorded as computational notes.  They indicate the next finite layer of the search and make the numerical reductions reproducible.  They are not used in the proof of the examples constructed here and do not enter the conditional theorem.

The main unconditional problem suggested by this work is the following.

\begin{problem}\label{prob:H-infinite}
Are there infinitely many squarefree integers $B$, all of whose prime factors exceed $101$, such that
\[
        797B-113322\dd(B)=1?
\]
\end{problem}

A positive answer would imply the infinitude of primary pseudoperfect numbers, because every such $B$ gives the primary pseudoperfect number $113322B$.  The port $H$ is a natural target: it already has the port-primitive fillings
\[
        149\cdot3109
        \qquad\text{and}\qquad
        157\cdot1979\cdot10093\cdot16879.
\]
At present no unconditional mechanism is known that produces infinitely many fillings of $H$ or of any other fixed port.  The finite $H_6$ problem recorded in Appendix~\ref{app:computational-notes} is the next computational layer in this direction.

\appendix

\section{Computational notes on the \texorpdfstring{$H_6$}{H6} layer}\label{app:computational-notes}

The material in this appendix records search reductions and open computational subproblems.  It is included to make the numerical discussion reproducible, but no theorem in the main text depends on these exploratory exclusions.  Statements in this appendix should therefore be read as computational notes unless explicitly promoted to a theorem elsewhere.

The natural next layer for the port $H$ is to ask for six-prime fillings
\begin{equation}\label{eq:H6}
        B=q_1q_2q_3q_4q_5q_6,
        \qquad
        q_i>101,
\end{equation}
with
\[
        797B-113322\dd(B)=1.
\]
Such a filling would give a primary pseudoperfect number with eleven prime factors.

The following subsections analyze the \emph{known inherited channels} arising from the displayed fillings $B_2$, $B_4$, and $B_5$.  This is not asserted to be a complete classification of all inherited $H_6$ channels: another as-yet-undiscovered $H$-filling with fewer than six prime factors would create an additional inherited channel.  The remaining alternative is a port-primitive $H_6$ filling in the sense of Definition~\ref{def:primitive-filling}.

\subsection{Known channel from \texorpdfstring{$B_2=149\cdot3109$}{B2 = 149*3109}}
If $B=B_2C$ and $\omega(B)=6$, then $\omega(C)=4$.  Since
\[
        113322B_2=52495396602=:K_7,
\]
we need
\[
        C-K_7\dd(C)=1,
        \qquad
        \omega(C)=4.
\]
This is a four-prime inheritance problem from the known primary pseudoperfect number $K_7$.

\subsection{Known channel from \texorpdfstring{$B_4$}{B4}}
If $B=B_4C$ and $\omega(B)=6$, then $\omega(C)=2$.  Since
\[
        113322B_4=N_9,
\]
we would need
\[
        C-N_9\dd(C)=1,
        \qquad
        \omega(C)=2.
\]
Equivalently, if $C=pq$, then
\[
        (p-N_9)(q-N_9)=N_9^2+1.
\]
The factorization
\[
        N_9^2+1
        =5\cdot22861\cdot34646497971913\cdot9085080009049858397
\]
leads to eight factor pairs.  None gives two primes $p,q$.  Hence no $H_6$ filling is inherited from $B_4$ by adding two primes.

\subsection{Known channel from \texorpdfstring{$B_5$}{B5}}
The filling
\[
        B_5=B_4(N_9+1)
\]
has five prime factors.  To obtain an $H_6$ filling by adding one prime, the added prime would have to be $N_{10}+1$, but $N_{10}+1$ is composite.  Hence no $H_6$ filling is inherited from $B_5$.

\subsection{Port-primitive and unknown-channel \texorpdfstring{$H_6$}{H6} fillings}
After the known channels above have been analyzed, the remaining targets are:
\begin{enumerate}[label=(\alph*)]
    \item the four-prime inheritance problem from $K_7=52495396602$ in the known $B_2$ channel;
    \item port-primitive six-prime fillings of $H$;
    \item any inherited channel arising from a smaller $H$-filling not yet found or not yet excluded.
\end{enumerate}
This wording is deliberately cautious: unless one proves that the displayed $B_2$, $B_4$, and $B_5$ exhaust all smaller $H$-fillings, the inherited-channel list above should be read as a list of known channels, not as a complete classification.

\section{Exploratory computational status for \texorpdfstring{$H_6$}{H6}}

This section records exploratory computations.  They are not used in the proofs above and are not presented as certified theorems.  The six-prime filling problem initially has $111$ possible first primes $q_1$, ranging from $149$ to $829$ after the elementary reciprocal-capacity pruning.  The present exploratory computation reduces the possible first prime to
\[
        q_1\le409.
\]
For the four largest remaining values
\[
        q_1\in\{389,397,401,409\},
\]
we enumerated all admissible four-prime prefixes and checked all $t\le1000$ cases in the discriminant~\eqref{eq:D-t}.  No square discriminant occurred.  The current status is summarized in Table~\ref{tab:H6-exploratory}.

\begin{center}
\begin{tabular}{rrrrrr}
\toprule
$q_1$ & four-prefixes & $T<0$ & $0\le T\le1000$ & checked $t$ & square hits \\
\midrule
409 & 1,156,527 & 1,047,472 & 98,398 & 6,043,695 & 0 \\
401 & 1,209,161 & 1,157,101 & 48,832 & 1,631,329 & 0 \\
397 & 1,215,077 & 1,194,620 & 19,190 & 609,648 & 0 \\
389 & 1,527,289 & 1,419,120 & 99,337 & 4,884,670 & 0 \\
\bottomrule
\end{tabular}
\captionof{table}{Exploratory search data for several high remaining first-prime branches of $H_6$.}
\label{tab:H6-exploratory}
\end{center}

These data do not exclude the branches completely.  They show only that any completion in the displayed branches must arise from a prefix with $T>1000$.

\section{Independent verification appendix}

This appendix separates the proof-critical arithmetic checks from the exploratory search data.  The main equalities in the paper are directly checkable from the displayed factorizations.  For machine verification, the source distribution also contains a companion archive with the following files:
\begin{itemize}
    \item \texttt{verify\_main\_claims.py}, checking the PPN identities, the Pocklington certificate for $p_{10}=N_9+1$, and the displayed factorizations;
    \item \texttt{pocklington\_certificates.json}, a machine-readable list of top-level Pocklington data;
    \item \texttt{search\_h6\_exploratory.py}, containing helper routines for the exploratory $H_6$ computations.
\end{itemize}
A permanent public deposit should accompany any circulated version of the preprint.  The archive used for this revision is \texttt{ppn\_verification\_archive.tar.gz}, with SHA256 checksum \texttt{\seqsplit{4d384f61d03eeb2b9e9f67252f1df22e70df55995c4e14eec5951285edde743b}}.  Until a permanent deposit is made, the proof-critical certificate data are summarized in Table~\ref{tab:cert-data} below.

\subsection{Top-level factorizations}
The factorizations used in the main text are:
\[
        p_{10}-1=N_9
        =2\cdot3\cdot11\cdot17\cdot101\cdot157\cdot1979\cdot10093\cdot16879,
\]
\[
\begin{aligned}
        N_{10}+1
        &=7\cdot37\cdot73\cdot407221\cdot2746750419901\cdot1701301706648581,
\end{aligned}
\]
\[
        N_9^2+1
        =5\cdot22861\cdot34646497971913\cdot9085080009049858397,
\]
\[
\begin{aligned}
        N_{10}^2+1 &=21807157\cdot480382349\cdot Q,\\
        Q &= {\scriptstyle 123572138719194583969192220095883252267503088389616114960309}.
\end{aligned}
\]
The verification archive checks that each factor displayed above is prime and that the products multiply to the asserted integers.  The JSON certificate file contains a recursive Pocklington-style certificate tree for these primes and the auxiliary primes arising in their $p-1$ factorizations.

\subsection{Top-level Pocklington data}
\begin{center}
\scriptsize
\resizebox{\textwidth}{!}{%
\begin{tabular}{p{0.30\textwidth}p{0.58\textwidth}r}
\toprule
prime $p$ & factorization of $p-1$ used at the top level & base \\
\midrule
\texttt{5998279018951962403} & $2\cdot3\cdot11\cdot17\cdot101\cdot157\cdot1979\cdot10093\cdot16879$ & $3$ \\
\texttt{407221} & $2^2\cdot3\cdot5\cdot11\cdot617$ & $2$ \\
\texttt{2746750419901} & $2^2\cdot3^2\cdot5^2\cdot23\cdot769\cdot172553$ & $7$ \\
\texttt{1701301706648581} & $2^2\cdot3\cdot5\cdot28355028444143$ & $6$ \\
\texttt{34646497971913} & $2^3\cdot3^2\cdot53\cdot373\cdot24341209$ & $5$ \\
\texttt{9085080009049858397} & $2^2\cdot11\cdot206479091114769509$ & $2$ \\
\texttt{21807157} & $2^2\cdot3\cdot7^2\cdot37087$ & $2$ \\
\texttt{480382349} & $2^2\cdot120095587$ & $2$ \\
\texttt{\seqsplit{123572138719194583969192220095883252267503088389616114960309}} & $2^2\cdot59\cdot$\texttt{\seqsplit{523610757284722813428780593626623950286030035549220826103}} & $2$ \\
\bottomrule
\end{tabular}%
}
\captionof{table}{Top-level Pocklington data for the primes used in the displayed factorizations.  The recursive primality of the auxiliary factors in the middle column is checked by the verification script.}
\label{tab:cert-data}
\end{center}

\subsection{Pocklington certificate for \texorpdfstring{$p_{10}$}{p10}}
The certificate used in Theorem~\ref{thm:p10-prime} is short enough to display.  With
\[
        p_{10}=5998279018951962403,
        \qquad
        a=3,
\]
one has
\[
        a^{p_{10}-1}\equiv1\pmod {p_{10}},
\]
and for every prime divisor
\[
        q\in\{2,3,11,17,101,157,1979,10093,16879\}
\]
of $p_{10}-1$, one has
\[
        \gcd\bigl(a^{(p_{10}-1)/q}-1,p_{10}\bigr)=1.
\]
The factorization of $p_{10}-1$ is complete, so this proves the primality of $p_{10}$ by Pocklington's criterion.

\subsection{Verification script excerpt}
The following short Python fragment verifies the central arithmetic identities and the Pocklington certificate used here.  The companion archive contains the longer recursive certificate generator.

\begin{lstlisting}[style=pythonsmall]
from math import prod, gcd
from fractions import Fraction
from sympy import isprime, factorint

P9 = [2,3,11,17,101,157,1979,10093,16879]
N9 = prod(P9)

assert N9 == 5998279018951962402
assert all(isprime(p) for p in P9)
assert 1 + sum(N9 // p for p in P9) == N9
assert Fraction(1, N9) + sum(Fraction(1, p) for p in P9) == 1

p10 = N9 + 1
assert p10 == 5998279018951962403
assert pow(3, p10 - 1, p10) == 1
for q in P9:
    assert gcd(pow(3, (p10 - 1)//q, p10) - 1, p10) == 1
assert isprime(p10)

N10 = N9 * p10
P10 = P9 + [p10]

assert N10 == 35979351189199316534587473905773572006
assert all(isprime(p) for p in P10)
assert 1 + sum(N10 // p for p in P10) == N10
assert Fraction(1, N10) + sum(Fraction(1, p) for p in P10) == 1

print(factorint(N10 + 1))
print(factorint(N10*N10 + 1))
\end{lstlisting}

\end{document}